\documentclass[a4paper,12pt]{article}

\usepackage{amssymb}
\usepackage{amsmath}
\usepackage{amsthm}%proofs.
\usepackage{color}
\usepackage{pdflscape}

\usepackage{mathptmx}
\usepackage{amsfonts}
\usepackage{graphicx}

\newtheorem{proposition}{Proposition}[section]
\newtheorem{remark}{Remark}[section]

\newcommand\blfootnote[1]{
  \begingroup
  \renewcommand\thefootnote{}\footnote{#1}
  \addtocounter{footnote}{-1}
  \endgroup
}

\title{A Shallow Water Model for Hydrodynamic Processes on
    Vegetated Hillslope.  Water Flow Modulus.}
\author{Stelian Ion, Dorin Marinescu, Stefan-Gicu Cruceanu\medskip
  \\{\small\it ``Gheorghe Mihoc-Caius Iacob'' Institute of Mathematical Statistics and Applied}
  \\{\small\it Mathematics, Romanian Academy, PO Box 1-24, 050711 Bucharest, Romania}
}
\date{}

\begin{document}
\maketitle

\blfootnote{{\it Email addresses:} stelian.ion@ima.ro (Stelian Ion), dorin.marinescu@ima.ro (Dorin Marinescu), stefan.cruceanu@ima.ro (Stefan-Gicu Cruceanu)}

\noindent
{\bf Abstract}
\medskip

\noindent
    The hillslope hydrological processes are very important
    in watershed hydrology research.  In this paper we focus
    on the water flow over the soil surface with vegetation
    in a hydrographic basin.  We introduce a PDE model based
    on general principles of fluid mechanics where the
    unknowns are water depth and water velocity.  The
    influence of the plant cover to the water dynamics is
    given by porosity (a quantity related to the density of
    the vegetation), which is a function defined over the
    hydrological basin.  Using finite volume method for
    approximating the spatial derivatives, we build an ODE
    system which constitutes the base of the discrete model
    we will work with.  We discuss and investigate several
    physical relevant properties of this model.  Finally, we
    use numerical results to validate the model.



\section{Introduction}
Mathematical modeling of the hydrodynamic processes in
hydrographic basins is of great interest.  The subject is
very rich in practical applications and there is not yet a
satisfactory model to enhance the entire complexity of these
processes.  However, there are plenty of performant models
dedicated to some specific aspects of the hydrodynamic
processes only.  To review the existent mathematical models
is beyond the purposes of this paper, but we can group them
into two large classes: physical base models and regression
models.  The most known regression models are the unit
hydrograph \cite{dooge} and universal soil loss \cite{rusle,
  wisch}.  From the first class, we mention here a few well
known models: SWAT \cite{swat}, SWAP \cite{swap} and
KINEROS \cite{kineros}.  Due to the complexity and
heterogeneity of the processes (see \cite{mcdonnel}), models
in this class are not purely physical because they need
additional empirical relations.  The main difference between
models here is given by the nature of the empirical
relations.  For example, in order to model the surface of
the water flow, SWAP and KINEROS use a mass balance equation
and a closure relation, while SWAT combines the mass balance
equation with the momentum balance equation.  A very special
class of models are cellular automata which combine
microscale physical laws with empirical closure relations in
a specific way to build up a macroscale model, e.g. CAESAR
\cite{caesar, sds-ose}.

In this paper, we introduce a physical model described by
shallow water type equations.  This model is obtained from
general principles of fluid mechanics using a space average
method and takes into consideration topography, water-soil
and water-plant interactions.  To numerically integrate the
equations, we first apply a finite volume method to
approximate the spatial derivatives and then use a type of
fractional time-step method to gain the evolution of the
water depth and velocity field.

After introducing the PDE model in Section
\ref{sect_ShalowWaterEquations}, we perform the Finite
Volume Method approximation in Section
\ref{sect_FVMapproximationof2Dmodel} and obtain an ODE
[Aversion of Shallow Water Equations.  In Section
\ref{sect_PropOfSemidiscreteScheme}, we investigate some
physical relevant qualitative properties of this ODE system:
monotonicity of the energy, positivity of the water depth
function $h$, well balanced properties of the scheme.  In
Section \ref{sect_FractionalSteptimeSchemes}, we obtain the
full discrete version of our continuous model; we tackle on
the validation method and give some numerical results in the
last section.

\section{Shalow Water Equations}
\label{sect_ShalowWaterEquations}
The model we discuss here is a simplified version a more
general model of water flow on a hillslope introduced in
\cite{imc-rap}.  Assume that the soil surface is
represented by
\begin{equation*}
  x^3=z(x^1,x^2), \quad (x^1,x^2)\in\Omega,
\end{equation*}
and the first derivatives of the function $z(\cdot,\cdot)$
are small quantities.  The unknown variables of the model
are the water depth $h(t,x)$ and the two components
$v_a(t,x)$ of the water velocity $\boldsymbol{v}$. The
density of the plant cover is quantified by a porosity
function $\theta(x)$.  The model reads as
\begin{equation}
  \label{swe_vegm_rm.02}
  \begin{array}{rl}
    \partial_t\theta h+\partial_b(\theta h v^b)= & \mathfrak{M},\\
    \partial_t(\theta hv_a)+\partial_b(hv_av^b)+\theta h\partial_aw= & 
       -{\cal K}(h,\theta)|\boldsymbol{v}|v_a, \quad a=1,2.
  \end{array}
\end{equation}
The term ${\cal K}(h,\theta)|\boldsymbol{v}|v_a$ quantifies
the interactions water-soil and water-plants \cite{baptist,
  nepf}. The function ${\cal K}(h,\theta)$ is given by
\begin{equation}
  \label{swe_vegm_rm.03} 
  {\cal K}(h,\theta) = \alpha_p h \left(1-\theta\right) + \theta \alpha_s,
\end{equation}
where $\alpha_p$ and $\alpha_s$ are two characteristic
parameters of the strength of the water-plant and water-soil
interactions, respectively.  The contribution of rain and
infiltration to the water mass balance is taken into account
by $\mathfrak{M}$.  In (\ref{swe_vegm_rm.02}),
$w=g\left[z(x^1,x^2)+h\right]$ stands for the free surface
level, and $g$ for the gravitational acceleration.

It is important to note that there is an energy function
${\cal E}$ given by
\begin{equation}
  \label{swe_vegm_rm.02-0}
  {\cal E} := \frac{1}{2}|\boldsymbol{v}|^2+g\left(x^3+\frac{h}{2}\right)
\end{equation}
that satisfies a conservative equation
\begin{equation}
  \label{swe_veg_numerics.08}
  \partial_t (\theta h{\cal E}) +
  \partial_b \left(\theta h v^b \left({\cal E}+g\frac{h}{2}\right)\right) =
  \mathfrak{M}\left(-\frac{1}{2}|\boldsymbol{v}|^2+w\right) -{\cal K}|\boldsymbol{v}|^3.
\end{equation}
In the absence of the mass source, the system preserves the
steady state of a lake
\begin{equation}
  \label{swe_veg_numerics.08-01}
  \partial_a(x^3+h)=0, \quad v_a=0, \quad a=1,2.
\end{equation}
The model (\ref{swe_vegm_rm.02}) is a hyperbolic system of
equations with source term, see \cite{imc-act}.

Among different features we ask from our approximation
scheme, we want the numerical solutions to preserve the
lake, the scheme to be well balanced and energetic
conservative.  These last two properties of a numerical
algorithm for the shallow water equation are very important,
especially for the case of hydrographic basin applications,
because they allow the lake formation and prevent the
numerical solution to oscillate in the neighborhood of a
lake.  In the absence of vegetation, one can find many such
schemes, see \cite{seguin, noelle, nordic}, for example.

\section{Finite Volume Method Approximation of 2D Model}
\label{sect_FVMapproximationof2Dmodel}
Let $\Omega$ be the domain of the space variables $x^1$,
$x^2$ and $\Omega=\cup_i \omega_i, i=\overline{1,N}$ an
admissible polygonal partition, \cite{veque}.  To build a
spatial discrete approximation of the model
(\ref{swe_vegm_rm.02}), one integrates the continuous
equations on each finite volume $\omega _i$ and then defines
an approximation of the integrals.

Let $\omega_i$ be an arbitrary element of the partition.
Relatively to it, the integral form of
(\ref{swe_vegm_rm.02}) reads as \def\msr#1{{\rm m(#1)}}
\begin{equation}
  \label{fvm_2D_eq.01}
  \begin{array}{rl}
    \displaystyle\partial_t\int\limits_{\omega_i}\theta h{\rm d}x+
    \int\limits_{\partial\omega_i}\theta h \boldsymbol{v}\cdot\boldsymbol{n}{\rm d}s=
    &\displaystyle\int\limits_{\omega_i}\mathfrak{M}{\rm d}x,\\
      \displaystyle\partial_t\int\limits_{\omega_i}\theta h v_a{\rm d}x+
      \int\limits_{\partial\omega_i}\theta h v_a\boldsymbol{v}\cdot\boldsymbol{n}{\rm d}s+
      \int\limits_{\omega_i}\theta h\partial_a w{\rm d}x=
      &\displaystyle-\int\limits_{\omega_i}{\cal K}|\boldsymbol{v}|v_a{\rm d}x, \quad a=1,2.
  \end{array}
\end{equation}

Now, we build a discrete version of the integral form by
introducing some quadrature formulas.  With $\psi_i$
standing for some approximation of $\psi$ on $\omega_i$, we
introduce the approximations
\begin{equation}
  \label{fvm_2D_eq.01-01}
  \int\limits_{\omega_i}\theta h{\rm d}x\approx\sigma_i\theta_ih_i,\quad
  \int\limits_{\omega_i}\theta h v_a{\rm d}x\approx\sigma_i\theta_ih_iv_{a\,i},\quad
  \int\limits_{\omega_i}{\cal K}|\boldsymbol{v}|v_a{\rm d}x\approx\sigma_i {\cal K}_i|\boldsymbol{v}|_iv_{a\,i},
\end{equation}
where $\sigma_i$ denotes the area of the polygon $\omega_i$.

For the integrals of the gradient of the free surface, we
start from the identity
\begin{equation}
  \label{fvm_2D_eq.01-01identity}
  \int\limits_{\omega_i}\theta h\partial_a w{\rm d}x
  =-\int\limits_{\omega_i} w \partial_a\theta h{\rm d}x+
   \int\limits_{\partial_i\omega_i} w \theta hn_a{\rm d}s.
\end{equation}
Assume that $w$ is constant and equal to $w_i$ on $\omega_i$
to approximate the first integral on the r.h.s. of
(\ref{fvm_2D_eq.01-01identity}). Then, we obtain:
\begin{equation}
  \label{fvm_2D_eq.01-02}
  \int\limits_{\omega_i}\theta h\partial_a w{\rm d}x\approx\int\limits_{\partial_i\omega_i} (w-w_i) \theta hn_a{\rm d}s.
\end{equation}
Note that if $\omega_i$ is a regular polygon and $w_i$ is
the cell-centered value of $w$, then the approximation is of
second order accuracy for smooth fields and it preserves the
null value in the case of constant fields $w$.

We introduce the notation
\begin{equation}
  \label{fvm_2D_eq.psi}
  \widetilde{\psi}|_{\partial \omega(i,j)}:=\int\limits_{\partial \omega(i,j)}\psi{\rm d}s.
\end{equation}

Using the approximations (\ref{fvm_2D_eq.01-01}) and
(\ref{fvm_2D_eq.01-02}) and keeping the boundary integrals,
one can write
\begin{equation}
  \label{fvm_2D_eq.02}
  \begin{array}{rl}
    \sigma_i\partial_t\theta_i h_i+
    \sum\limits_{j\in{\cal N}(i)}\widetilde{\theta h v_n}|_{\partial \omega(i,j)}=&\sigma_i\mathfrak{M}_i,\\
    \sigma_i\partial_t\theta_i h_i v_{a\,i}+
    \sum\limits_{j\in{\cal N}(i)}\widetilde{\theta h v_a v_n}|_{\partial \omega(i,j)}+ 
    \sum\limits_{j\in{\cal N}(i)}\widetilde{(w-w_i)\theta h}n_a|_{\partial \omega(i,j)}
    =&-\sigma_i{\cal K}_i|\boldsymbol{v}|_iv_{a\,i},
  \end{array}
\end{equation}
where ${\cal N}(i)$ denotes the set of all the neighbors of
$\omega_i$ and $\partial\omega(i,j)$ is the common boundary
of $\omega_i$ and $\omega_j$.

The next step is to define the approximations of the
boundary integrals in (\ref{fvm_2D_eq.02}). We approximate
an integral $\widetilde{\psi}|_{\partial \omega(i,j)}$ of
the form (\ref{fvm_2D_eq.psi}) by considering the integrand
$\psi$ to be a constant function
$\psi_{(i,j)}(\psi_i,\psi_j)$, where $\psi_i$ and $\psi_j$
are some fixed values of $\psi$ on the adjacent cells
$\omega_i$ and $\omega_j$, respectively.  Thus,
\begin{equation}
  \label{fvm_2D_eq.03}
  \begin{array}{l}
    \widetilde{\theta h v_n}|_{\partial \omega(i,j)}\approx l_{(i,j)}\theta h_{(i,j)} (v_n)_{(i,j)},\\
    \widetilde{\theta h v_a v_n}|_{\partial \omega(i,j)}\approx l_{(i,j)}\theta h_{(i,j)} (v_a)_{(i,j)}(v_n)_{(i,j)},\\
    \widetilde{(w-w_i)\theta h}n_a|_{\partial \omega(i,j)} \approx l_{(i,j)}(w_{(i,j)}-w_i) \theta h^s_{(i,j)}(n_a)_{(i,j)},
   \end{array}
\end{equation} 
where $\boldsymbol{n}_{(i,j)}$ denotes the unitary normal to
the common side of $\omega_i$ and $\omega_j$ pointing
towards $\omega_j$, and $l_{(i,j)}$ is the length of this
common side.

The issue is to define the interface value functions
$\psi_{(i,j)}(\psi_i,\psi_j)$ so that the resulting scheme
is well balanced and energetically stable.

\medskip\noindent
{\bf Well balanced and energetically stable scheme.}  For any
internal interface $(i,j)$, we define the following
quantities:
\begin{equation}
  \label{fvm_2D_eq.04}
  \begin{array}{l}
    (v_a)_{(i,j)}=\displaystyle\frac{v_{a\,i}+v_{a\,j}}{2}, \quad a=1,2,\\
    (v_n)_{(i,j)}=\boldsymbol{v}_{(i,j)}\cdot\boldsymbol{n}_{(i,j)},\\
    w_{(i,j)}=\displaystyle\frac{w_{i}+w_{j}}{2},
  \end{array}
\end{equation}
and 
\begin{equation}
  \label{fvm_2D_eq.05}
  \theta h^s_{(i,j)}=
  \left\{
    \begin{array}{ll}
      \theta h_{(i,j)}, & {\rm if}\; (v_n)_{(i,j)}\neq 0,\\
      \theta_i h_i, & {\rm if}\; (v_n)_{(i,j)}=0 \;{\rm and}\; w_i>w_j,\\
      \theta_j h_j, & {\rm if}\; (v_n)_{(i,j)}=0 \;{\rm and}\; w_i\leq w_j.
    \end{array}
  \right.
\end{equation}

\medskip\noindent
{\bf $h$-positivity.}  In order to preserve the positivity
of $h$, we define $\theta h_{(i,j)}$ as
\begin{equation}
  \label{fvm_2D_eq.06}
  \theta h_{(i,j)}=
  \left\{
    \begin{array}{ll}
      \theta_i h_i, & {\rm if}\; (v_n)_{(i,j)}>0,\\
      \theta_j h_j, & {\rm if}\; (v_n)_{(i,j)}<0. 
    \end{array}
  \right.
\end{equation}
The semidiscrete scheme takes now the form of the following
system of ODEs
\begin{equation}
  \label{fvm_2D_eq.07}
  \begin{array}{rl}
    \sigma_i\displaystyle\frac{\rm d}{{\rm d}t}\theta_i h_i+\sum\limits_{j\in{\cal N}(i)}l_{(i,j)}\theta h_{(i,j)} (v_n)_{(i,j)}=&\sigma_i\mathfrak{M}_i,\\
    \sigma_i\displaystyle\frac{\rm d}{{\rm d}t}\theta_i h_i v_{a\,i}+\sum\limits_{j\in{\cal N}(i)}l_{(i,j)}\theta h_{(i,j)} (v_a)_{(i,j)} (v_n)_{(i,j)}+&\\ 
    +\displaystyle\frac{1}{2}\sum\limits_{j\in{\cal N}(i)}l_{(i,j)}(w_j-w_i)(\theta h)^s_{(i,j)} n_a|_{(i,j)}&=-\sigma_i{\cal K}_i|\boldsymbol{v}|_iv_{a\,i}.
  \end{array}
\end{equation}

\medskip\noindent
{\bf Boundary conditions. Free discharge.}  We need to
define the values of $h$ and $\boldsymbol{v}$ on the
external sides of $\Omega$.  For each side in
$\Gamma=\partial\Omega$ we introduce a new cell (``ghost''
element) adjacent to the polygon $\omega_i$ corresponding to
that side.  For each ``ghost'' element, one must somehow
define its altitude and then we set zero values to its water
depth.  We can now define $h$ and $\boldsymbol{v}$ on the
external sides of $\Omega$ by
\begin{equation}
  \label{fvm_2D_eq.070}
  \begin{array}{l}
    \boldsymbol{v}_{\partial \omega_i\cap \Gamma}=\boldsymbol{v}_i,\\
    h_{\partial \omega_i\cap \Gamma}=
    \left\{
    \begin{array}{ll}
      h_i, & {\rm if}\; \boldsymbol{v}_i\cdot \boldsymbol{n}|_{\partial \omega_i\cap \Gamma }>0,\\
      0, & {\rm if}\; \boldsymbol{v}_i\cdot \boldsymbol{n}|_{\partial \omega_i\cap \Gamma }<0.
    \end{array}
    \right.
  \end{array}
\end{equation}
Now, the solution is sought inside the positive cone
$h_i>0, \; i=\overline{1,N}$.

\section{Properties of the semidiscrete scheme}
\label{sect_PropOfSemidiscreteScheme}
The ODE model (\ref{fvm_2D_eq.07}) can have discontinuities
in the r.h.s. and therefore it is possible that the solution
in the classical sense of this system might not exist for some
initial data.  However, the solution in Filipov sense
\cite{filipov} exists for any initial data.

There are initial data for which the solution in the
classical sense exists only locally in time.  Since the
numerical scheme is a time approximation of the semidiscrete
form (\ref{fvm_2D_eq.07}), it is worthwhile to analyze the
properties of these classical solutions.  Numerical schemes
preserving properties of some particular solutions of the
continuum model were and are intensively investigated in the
literature \cite{bouchut-book, seguin, nordic,
  well-balanced}.  In the present section we investigate
such properties for the semidiscrete scheme and the next
section is dedicated to the properties of the complete
discretized scheme.

\subsection{Energy balance}
Definition (\ref{fvm_2D_eq.04}) yields a dissipative
conservative equation for the cell energy ${\cal E}_i$,
\begin{equation}
  \label{fvm_2D_eq.060}
  {\cal E}_i(h_i,\boldsymbol{v}_i)=\theta_i\left(\frac{1}{2}{|\boldsymbol{v}|^2_i}{h_i}+\frac{1}{2}gh^2_i+gx^3_ih_i\right). 
\end{equation}
The time derivative of ${\cal E}_i$ can be written as
\begin{equation}
  \label{fvm_2D_eq.0700}
  \sigma_i \displaystyle\frac{{\rm d }{{\cal E}_i} }{ {\rm d} t} = \sigma_i
  \left(
    \left(w_i-\frac{1}{2}|\boldsymbol{v}|^2_i\right) \displaystyle\frac{{\rm d }{\theta_i h_i}}{ {\rm d} t} 
    +\left< \boldsymbol{v}_i, \displaystyle\frac{{\rm d }{\theta_ih_i\boldsymbol{v}_i}}{ {\rm d} t} \right>
  \right),
\end{equation}
where $\left<\cdot,\cdot\right>$ denotes the euclidean
scalar product in $\mathbb{R}^2$.

\begin{proposition}[Cell energy equation]
  \label{cell_energy}
  In the absence of mass source, one has
  \begin{equation}
    \label{fvm_2D_eq.071}
    \sigma_i\displaystyle\frac{\rm d}{{\rm d}t}{\cal E}_i+\sum\limits_{j\in{\cal N}(i)}l_{(i,j)}\left<{\cal H}_{(i,j)},\boldsymbol{n}_{(i,j)}\right>=-\sigma_i{\cal K}_i|\boldsymbol{v}|^3_i,
  \end{equation}
  where
  \begin{equation*}
    {\cal H}_{(i,j)}=\frac{1}{2}\theta h_{(i,j)}
    \left(
      w_i\boldsymbol{v}_i+w_j\boldsymbol{v}_j+\left<\boldsymbol{v}_i,\boldsymbol{v}_j\right>\boldsymbol{v}_{(i,j)}
    \right).
  \end{equation*}
\end{proposition}

\begin{remark}
  If $(\theta h,v,w)_j=(\theta h,v,w)_i$ for any
  $j\in{\cal N}(i)$, then
  \begin{equation*}
    {\cal H}=\theta h \boldsymbol{v}\left(\frac{1}{2} |\boldsymbol{v}|^2+w\right)
  \end{equation*}
  is the continuous flux energy in {\rm (\ref{swe_veg_numerics.08})}.
\end{remark}

\begin{proof}
Using the equality
(\ref{fvm_2D_eq.0700}), we can write
\begin{equation*}
  \begin{array}{rcl}
    \sigma_i\displaystyle\frac{\rm d}{{\rm d}t}{\cal E}_i&=&-(w_i-\displaystyle\frac{1}{2}|\boldsymbol{v}|_i^2)\sum\limits_{j\in{\cal N}(i)}l_{(i,j)}\theta h_{(i,j)} (v_n)_{(i,j)}-\\
    &&-\left<
    \boldsymbol{v}_i,\sum\limits_{j\in{\cal N}(i)}l_{(i,j)}\theta h_{(i,j)} \boldsymbol{v}_{(i,j)} (v_n)_{(i,j)}
      \right>-\\
    &&-\displaystyle\frac{1}{2}\left<\boldsymbol{v}_i,\sum\limits_{j\in{\cal N}(i)}l_{(i,j)}(w_j-w_i)(\theta h)^s_{(i,j)} \boldsymbol{n}_{(i,j)}\right>-\\
    &&-\sigma_i{\cal K}_i|\boldsymbol{v}|_i^3.
  \end{array}
\end{equation*}
Now, one has the identities
\begin{equation*}
  \begin{array}{rcl}
    w_i\displaystyle\sum\limits_{j\in{\cal N}(i)}l_{(i,j)}\theta h_{(i,j)} (v_n)_{(i,j)}
    &=&\displaystyle\sum\limits_{j\in{\cal N}(i)}l_{(i,j)}\theta h_{(i,j)} (v_n)_{(i,j)}\displaystyle\frac{w_i+w_j}{2}+\\
    &&+\displaystyle\sum\limits_{j\in{\cal N}(i)}l_{(i,j)}\theta h_{(i,j)} (v_n)_{(i,j)}\displaystyle\frac{w_i-w_j}{2}
  \end{array}
\end{equation*}
and
\begin{equation*}
  \begin{array}{l}
    \left<\boldsymbol{v}_i,\displaystyle\sum\limits_{j\in{\cal N}(i)}l_{(i,j)}(w_j-w_i)(\theta h)^s_{(i,j)} \boldsymbol{n}_{(i,j)}\right>=\\
    =\displaystyle\sum\limits_{j\in{\cal N}(i)}l_{(i,j)}(w_j-w_i)(\theta h)^s_{(i,j)}\left<\displaystyle\frac{\boldsymbol{v}_i+\boldsymbol{v}_j}{2}+\displaystyle\frac{\boldsymbol{v}_i-\boldsymbol{v}_j}{2}, \boldsymbol{n}_{(i,j)}\right>.
  \end{array}
\end{equation*}
Therefore
\begin{equation*}
  \begin{array}{r}
    w_i\displaystyle\sum\limits_{j\in{\cal N}(i)}l_{(i,j)}\theta h_{(i,j)} (v_n)_{(i,j)}+
    \displaystyle\frac{1}{2}\left<\boldsymbol{v}_i,\displaystyle\sum\limits_{j\in{\cal N}(i)}l_{(i,j)}(w_j-w_i)(\theta h)^s_{(i,j)} \boldsymbol{n}_{(i,j)}\right>=\\
    =\displaystyle\sum\limits_{j\in{\cal N}(i)}l_{(i,j)}\theta h_{(i,j)}\left<w_i\boldsymbol{v}_i+w_j\boldsymbol{v}_j,\boldsymbol{n}_{(i,j)}\right>.
  \end{array}
\end{equation*}
Similarly, one obtains the identity
\begin{equation*}
  \begin{array}{r}
    -\displaystyle\frac{1}{2}|\boldsymbol{v}|_i^2\sum\limits_{j\in{\cal N}(i)}l_{(i,j)}\theta h_{(i,j)} (v_n)_{(i,j)}+
    \left<\boldsymbol{v}_i,\sum\limits_{j\in{\cal N}(i)}l_{(i,j)}\theta h_{(i,j)} \boldsymbol{v}_{(i,j)} (v_n)_{(i,j)}\right>=\\
    =\displaystyle\frac{1}{2}\sum\limits_{j\in{\cal N}(i)}l_{(i,j)}\theta h_{(i,j)}\left<\boldsymbol{v}_i,\boldsymbol{v}_j\right>\left<\displaystyle\frac{\boldsymbol{v}_i+\boldsymbol{v}_j}{2}, \boldsymbol{n}_{(i,j)}\right>.
  \end{array}
\end{equation*}
\end{proof}

Taking out the mass exchange through the boundary, the
definitions of the interface values ensure the monotonicity
of the energy with respect to time.

\subsection{h-positivity and critical points}
\begin{proposition}[h-positivity]
  The ODE system {\rm (\ref{fvm_2D_eq.07})} with {\rm
    (\ref{fvm_2D_eq.04})}, {\rm (\ref{fvm_2D_eq.05})}, {\rm
    (\ref{fvm_2D_eq.06})} and {\rm (\ref{fvm_2D_eq.070})}
  preserves the positivity of the water depth function $h$.
\end{proposition}

\begin{proof}
One can rewrite the mass balance equations as
\begin{equation*}
  \sigma_i\displaystyle\frac{\rm d}{{\rm d}t}\theta_i h_i=
  -(\theta h)_i\sum\limits_{j\in{\cal N}(i)}l_{(i,j)} (v_n)^{+}_{(i,j)}+
  \sum\limits_{j\in{\cal N}(i)}l_{(i,j)}(\theta h)_{j} (v_n)^{-}_{(i,j)}.
\end{equation*}
Observe that if $h_i=0$ for some $i$, then
$\sigma_i\displaystyle\frac{\rm d}{{\rm d}t}\theta_i h_i\geq
0$.
\end{proof}

There are two kinds of stationary points for the ODE model:
the lake and uniform flow on an infinitely extended plan
with constant vegetation density.

\begin{proposition}[Stationary point. Uniform flow.]
  \label{river}
  Consider $\{\omega_i\}_{i=\overline{1,N}}$ to be a regular
  partition of $\Omega$ with $\omega_i$ regular polygons.
  Let $z-z_0=\xi_b x^b$ be a representation of the soil
  plane surface.  Assume that the discretization of the soil
  surface is given by
  \begin{equation}
    \label{fvm_2D_eq.08}
    z_i-z_0=\xi_b \overline{x}^b_i,
  \end{equation}
  where $\overline{x}^b_i$ is the mass center of the
  $\omega_i$ and $\theta_i=\theta$.  Then, given a value
  $h$, there is $\boldsymbol{v}$ so that the state
  $(h_i,\boldsymbol{v}_i)=(h,\boldsymbol{v})$,
  $i=\overline{1,N}$ is a stationary point of the ODE {\rm
    (\ref{fvm_2D_eq.07})}.
\end{proposition}

\begin{proof}
  For any constant state $h_i=h$ and $(v_a)_i=v_a$, the ODE
  (\ref{fvm_2D_eq.07}) reduces to
  \begin{equation*}
    \displaystyle\frac{1}{2}\theta h g\sum\limits_{j\in{\cal
        N}(i)}l_{(i,j)}(z_j-z_i) n_a|_{(i,j)} =-\sigma{\cal
      K}|\boldsymbol{v}|v_{a}.
  \end{equation*}
  Introducing the representation (\ref{fvm_2D_eq.08}), one
  writes
  \begin{equation*}
    \displaystyle\frac{1}{2}\theta h g\sum\limits_{j\in{\cal N}(i)}l_{(i,j)}\xi_b(\overline{x}^b_j-\overline{x}^b_i) n_a|_{(i,j)}
    =-\sigma{\cal K}|\boldsymbol{v}|v_{a}.
  \end{equation*}
  Note that for a regular partition one has the identity
  \begin{equation*}
    \overline{x}^b_j-\overline{x}^b_i=2(y_{(i,j)}-\overline{x}^b_i),
  \end{equation*}
  where $y_{(i,j)}$ is the midpoint of the common side
  $\overline{\omega}_i\cap\overline{\omega}_j$.  Taking into
  account that
  \begin{equation*}
    \begin{array}{ll}
      \displaystyle\frac{1}{2}\theta h g
      \sum\limits_{j\in{\cal N}(i)}l_{(i,j)}(z_j-z_i)
      n_a|_{(i,j)} & =\displaystyle\theta h g
                     \sum\limits_{j\in{\cal N}(i)}l_{(i,j)}\xi_b y^b_{(i,j)} n_a|_{(i,j)}\\
                   &=\displaystyle\theta h g
                     \int\limits_{\partial \omega_i}\xi_b x^b(s) n_a(s){\rm d}s\\
                   &=\displaystyle\theta h g
                     \int\limits_{\omega_i}\xi_b \partial_a x^b{\rm d}x\\
                   &=\sigma \theta h g \xi_a,
    \end{array}
  \end{equation*}
  we obtain that the velocity is a constant field
  \begin{equation}
    \label{fvm_2D_eq.09}
    v_a=\xi_a\left(\frac{\theta h g}{{\cal K} |\xi|}\right)^{1/2}.
  \end{equation}
\end{proof}

A lake is a stationary point characterized by a constant
value of the free surface and a null velocity field over
connected regions.  A lake for which $h_i>0$ for any
$i\in\{1,2,\ldots, N\}$ will be named {\it regular
  stationary point} and a lake that occupies only a part of
a domain flow will be named {\it singular stationary point}.

\begin{proposition}[Stationary point. Lake.]
  \label{lake}
  In the absence of mass source, the following properties hold:

  {\rm (a)} Regular stationary point: the state
  \begin{equation*}
    w_i=w \;\; \& \;\; \boldsymbol{v}_i=0, \; \forall i=\overline{1,N}
  \end{equation*}
  is a stationary point of ODE {\rm (\ref{fvm_2D_eq.07})}.

  {\rm (b)} Singular stationary point: the state
  \begin{equation*}
    \boldsymbol{v}_i=0, \; \forall i=\overline{1,N} \;\; \& \;\;
    w_i=w, \; \forall i\in {\cal I} \;\; \& \;\; h_i=0, \;
    z_i>w, \; \forall i\in \complement{\cal I},
  \end{equation*}
  for some ${\cal I}\subset \{1,2,\ldots,N\}$ is a
  stationary point. ($\complement{\cal I}$ is the complement
  of ${\cal I}$.)
\end{proposition}

\begin{proof}
  For the sake of simplicity, in the case of the singular
  stationary point, we consider that
  $\displaystyle\Omega_{\cal I}=\cup_{i\in{\cal I}}\omega_i$
  is a connected domain.  Since the velocity field is zero,
  it only remains to verify that
  \begin{equation*}
    \sum\limits_{j\in{\cal N}(i)}l_{(i,j)}(w_j-w_i)(\theta
    h)^s_{(i,j)} n_a|_{(i,j)}=0,
  \end{equation*}
  for any cell $\omega_i$.  If $i \in \complement{\cal I}$,
  then the above sum equals zero since $h^s_{(i,j)}=0$, for
  all $j\in{\cal N}(i)$.  If $i \in {\cal I}$, then the sum
  is again zero because either $h^s_{(i,j)}=0$, for
  $j\in\complement{\cal I}$ or $w_j=w_i$, for
  $j\in{\cal I}$.
\end{proof}

\section{Fractional Step-time Schemes}
\label{sect_FractionalSteptimeSchemes}
In what follows we discuss different explicit or
semi-implicit schemes in order to integrate the ODE
(\ref{fvm_2D_eq.07}).

We introduce some notations
\begin{equation}
  \label{fvm_2D_eq_frac.01}
  \begin{array}{ll}
    {\cal J}_{a\,i}(h,\boldsymbol{v}):=&-\displaystyle\sum\limits_{j\in{\cal N}(i)}l_{(i,j)}\theta h_{(i,j)} (v_a)_{(i,j)} (v_n)_{(i,j)},\\
    {\cal S}_{a\,i}(h,w):=&-\displaystyle\frac{1}{2}\sum\limits_{j\in{\cal N}(i)}l_{(i,j)}(w_j-w_i)(\theta h)^{s}_{(i,j)} n_a|_{(i,j)},\\
    {\cal L}_i((h,\boldsymbol{v})):=&-\displaystyle\sum\limits_{j\in{\cal N}(i)}l_{(i,j)}\theta h_{(i,j)} (v_n)_{(i,j)}.
  \end{array}
\end{equation}

Now, (\ref{fvm_2D_eq.07}) becomes
\begin{equation}
  \label{fvm_2D_eq_frac.02}
  \begin{array}{rl}
    \sigma_i\displaystyle\frac{\rm d}{{\rm d}t}\theta_i h_i&={\cal L}_i(h,\boldsymbol{v})+\sigma_i\mathfrak{M}(t,h),\\
    \sigma_i\displaystyle\frac{\rm d}{{\rm d}t}\theta_i h_i v_{a\,i}&={\cal J}_{a\,i}(h,\boldsymbol{v})+{\cal S}_{a\,i}(h,w)-{\cal K}(h)|\boldsymbol{v}_i| v_{a\,i}.
  \end{array}
\end{equation}

\noindent {\bf Source mass.} We assume that the source mass
$\mathfrak{M}$ is of the form
\begin{equation}
  \label{fvm_2D_eq.14}
  \mathfrak{M}(x,t,h)=r(t)-\theta(x)\iota(t,h),
\end{equation}
where $r(t)$ quantifies the rate of the rain and
$\iota(t,h)$ quantifies the infiltration rate.  The
infiltration rate is a continuous function and satisfies the
following condition
\begin{equation}
  \label{fvm_2D_eq.15}
  \iota(t,h)<\iota_{m},\quad {\rm if}\; h\geq 0.
\end{equation}
The basic idea of a fractional time method is to split the
initial ODE into two sub-models, integrate them separately,
and then combine the two solutions \cite{veque-phd, strang}.

We split the ODE (\ref{fvm_2D_eq.07}) into
\begin{equation}
  \label{fvm_2D_eq_frac.03}
  \begin{array}{rl}
    \sigma_i\displaystyle\frac{\rm d}{{\rm d}t}\theta_i h_i&={\cal L}_i(h,\boldsymbol{v}),\\
    \sigma_i\displaystyle\frac{\rm d}{{\rm d}t}\theta_i h_i v_{a\,i}&={\cal J}_{a\,i}(h,\boldsymbol{v}) +{\cal S}_{a\,i}(h,w),
  \end{array}
\end{equation}
and
\begin{equation}
  \label{fvm_2D_eq_frac.04}
  \begin{array}{rl}
    \sigma_i\displaystyle\frac{\rm d}{{\rm d}t}\theta_i h_i&=\sigma_i\mathfrak{M}_i(t,h),\\
    \sigma_i\displaystyle\frac{\rm d}{{\rm d}t}\theta_i h_i v_{a\,i}&=-{\cal K}(h)|\boldsymbol{v}_i| v_{a\,i}.
  \end{array}
\end{equation}
A first order fractional step time accuracy reads as
\begin{equation}
  \label{fvm_2D_eq_frac.05}
  \begin{array}{rl}
    \sigma(\theta h)^{*}&=\sigma(\theta h)^n+\triangle t_n{\cal L}((h,\boldsymbol{v})^{n}),\\
    \sigma(\theta hv_a)^{*}&=\sigma(\theta hv_a)^{n}+\triangle t_n \left({\cal J}_a((h,\boldsymbol{v})^{n})+{\cal S}_a((h,w)^{n})\right),\\
  \end{array}
\end{equation}
\begin{equation}
  \label{fvm_2D_eq_frac.06}
  \begin{array}{rl}
    \sigma(\theta h)^{n+1}&=\sigma(\theta h)^{*}+\sigma\triangle t_n\mathfrak{M}(t^{n+1},h^{n+1}),\\
    \sigma(\theta hv_a)^{n+1}&=\sigma(\theta hv_a)^{*}-\triangle t_n{\cal K}(h)|\boldsymbol{v}^{n+1}| v^{n+1}_{a}.\\
  \end{array}
\end{equation}
The steps (\ref{fvm_2D_eq_frac.05}) and
(\ref{fvm_2D_eq_frac.06}) lead to
\begin{equation}
  \label{fvm_2D_eq_frac.07}
  \begin{array}{rl}
    \sigma(\theta h)^{n+1}=&\sigma(\theta h)^n+\triangle t_n{\cal L}((h,\boldsymbol{v})^{n})+\sigma\triangle t_n\mathfrak{M}(t^{n+1},h^{n+1}),\\
    \sigma(\theta hv_a)^{n+1}=&\sigma(\theta hv_a)^{n}+\triangle t_n\left({\cal J}_a((h,\boldsymbol{v})^{n})+{\cal S}_a((h,w)^{n})\right)-\\
    &-\triangle t_n\sigma{\cal K}(h)|\boldsymbol{v}^{n+1}| v^{n+1}_{a}.\\
  \end{array}
\end{equation} 
To advance a time step, one needs to solve a scalar
nonlinear equation for $h$ and a 2D nonlinear system of
equations for velocity $\boldsymbol{v}$.

In what follows, we investigate some important physical
properties of the numerical solution given by
(\ref{fvm_2D_eq_frac.07}): $h$-positivity, well balanced
property and monotonicity of the energy.

\subsection{h-positivity. Stationary points}
\begin{proposition}[$h$-positivity]
  There exists an upper bound $\tau_n$ for the time step
  $\triangle t_n$ such that if $\triangle t_n<\tau_n$ and
  $h^n>0$, then $h^{n+1}\geq 0$.
\end{proposition}

\noindent
\begin{proof} For any cell $i$ one has
\begin{equation*}
  \begin{array}{ll}
    \sigma_i\theta_ih^{n+1}_i+\triangle t_n\iota(t^{n+1},h^{n+1}_i)=&\sigma_i\theta_ih^{n}_i\left(1-\displaystyle\frac{\triangle t_n}{\sigma_i}\sum\limits_{j\in{\cal N}(i)}l_{(i,j)} (v_n)^{n,+}_{(i,j)}
      \right)+\\
    &\displaystyle +\triangle t_n \sum\limits_{j\in{\cal N}(i)}l_{(i,j)}(\theta h^n)_{j} (v_n)^{n,-}_{(i,j)}+\triangle t_nr(t^{n+1}).
  \end{array}
\end{equation*}

A choice for the upper bound $\tau_n$ is given by
\begin{equation}
  \label{fvm_2D_eq_frac.07-1}
  \tau_n=\displaystyle\frac{1}{v^n_{\rm max}}\min_i\left\{\displaystyle\frac{\sigma_i}{\sum\limits_{j\in{\cal N}(i)}l_{(i,j)}}\right\}.
\end{equation}
\end{proof}

\begin{proposition}[Well balanced]
  The lake and the uniform flow are stationary points of the
  scheme {\rm (\ref{fvm_2D_eq_frac.07})}.
\end{proposition}

\noindent
\begin{proof}
  One can prove this result similarly as in propositions
  \ref{river} and \ref{lake}.
\end{proof}

Unfortunately, the semi-implicit scheme
(\ref{fvm_2D_eq_frac.07}) does not preserve the monotonicity
of the energy.

\subsection{Discrete energy}   
The variation of the energy between two consecutive time
steps can be written as
\begin{equation}
  \label{fvm_2D_eq_frac.08}
  \begin{array}{rl}
    {\cal E}^{n+1}-{\cal E}^n&=\displaystyle\sum\limits_i\theta_i\sigma_i(h_i^{n+1}-h^n)(w_i^n-\displaystyle\frac{|\boldsymbol{v}^n_i|^2}{2})+\\
    &+\displaystyle\sum\limits_i\theta_i\sigma_i\left<(h\boldsymbol{v})^{n+1}_i-(h\boldsymbol{v})^{n}_i,\boldsymbol{v}^n_i\right>+\\
    &+g\displaystyle\sum\limits_i\theta_i\sigma_i\displaystyle\frac{(h_i^{n+1}-h^n)^2}{2}+\sum\limits_i\theta_i\sigma_i \displaystyle\frac{h_i^{n+1}}{2}\left|\boldsymbol{v}^{n+1}_i-\boldsymbol{v}^{n}_i\right|^2.
  \end{array}
\end{equation}
If the sequence $(h,\boldsymbol{v})^n$ is given by the
scheme (\ref{fvm_2D_eq_frac.07}), we obtain
\begin{equation}
  \label{fvm_2D_eq_frac.09}
  \begin{array}{rl}
    {\cal E}^{n+1}-{\cal E}^n&=-\displaystyle\triangle t_n\sum\limits_i\sigma_i{\cal K}(h^{n+1})|\boldsymbol{v}^{n+1}_i|^2\left<\boldsymbol{v}_i^{n+1},\boldsymbol{v}^{n}_i\right>+\\
    &+\displaystyle g\sum\limits_i\theta_i\sigma_i\displaystyle\frac{(h_i^{n+1}-h^n)^2}{2}+\sum\limits_i\theta_i\sigma_i\displaystyle\frac{h_i^{n+1}}{2}\left|\boldsymbol{v}^{n+1}_i-\boldsymbol{v}^{n}_i\right|^2+\\
    &+TS+TB,
  \end{array}
\end{equation}                          
where $TB$ and $TS$ stand for the contribution of boundary
and mass source to the energy production.

Note that, in the absence of $TB$ and $TS$ we cannot
conclude from (\ref{fvm_2D_eq_frac.09}) that the energy is
decreasing in time.  Our numerical computations emphasize
that the scheme introduces spurious oscillations in the
neighborhood of the lake points.  In order to decrease a
possible increase of energy added by the semi-implicit
scheme (\ref{fvm_2D_eq_frac.07}) and to eliminate these
oscillations, we introduce an artificial viscosity in the
scheme \cite{veque, kurganov}.  Adding a ``viscous''
contribution to the term ${\cal J}$,
\begin{equation}
  \label{fvm_2D_eq_frac.10}
  {\cal J}^{\boldsymbol{v}}_{a\,i}={\cal J}_{a\,i}(h,\boldsymbol{v})+
  \displaystyle\sum\limits_{j\in{\cal N}(i)}l_{(i,j)}\mu_{(i,j)} ((v_a)_j- (v_a)_i),
\end{equation}
the variation of energy is now given by
\begin{equation}
  \label{fvm_2D_eq_frac.11}
  {\cal E}^{n+1}_{\boldsymbol{v}}-{\cal E}_{\boldsymbol{v}}^n={\cal E}^{n+1}-{\cal E}^n-\triangle t_n\displaystyle\sum\limits_{s(i,j)}l_{(i,j)}\mu_{(i,j)}\left|\boldsymbol{v}_i-\boldsymbol{v}_j\right|^2,
\end{equation}
where $\mu_{(i,j)}>0$ is the artificial viscosity.

\subsection{Stability}
The stability of any numerical scheme ensures that errors in
data at a time step are not further amplified along the next
steps.  To acquire the stability of our scheme, we have
investigated several time-bounds $\tau_n$ and different
formulas for the viscosity $\nu$.  The best results were
obtained with
\begin{equation}
  \label{fvm_2D_eq_frac.12}
  \tau_n=\displaystyle\frac{\phi_{\rm min}}{c^n_{\rm max}}, \quad
  \mu_{(i,j)}=(\theta h)_{(i,j)}c_{(i,j)},
\end{equation}
where
\begin{equation}
  \label{fvm_2D_eq_frac.13}
  \begin{array}{l}
    c_i=|\boldsymbol{v}|_i+\sqrt{gh_i},\\
    c_{\rm max}=\max\limits_i \{c_i\},\\
    c_{(i,j)}=\max\{c_i,c_j\},\\
    \phi_{\rm min}=\min\limits_i
    \left\{
    \displaystyle\frac{\sigma_i}{\sum\limits_{j\in{\cal N}(i)}l_{(i,j)}}
    \right\}.
  \end{array}
\end{equation}

\begin{remark}
  An upper bound, as {\rm (\ref{fvm_2D_eq_frac.12})}, for
  the time-step is well known in the theory of hyperbolic
  system, CFL condition {\rm \cite{bouchut-book, veque}}.
\end{remark}

\section{Validation}
A rough classification of validation methods splits them
into two classes: internal and external.  For the internal
validation, one analyses the numerical results into a
theoretical frame: comparison to analytical results,
sensibility to the variation of the parameters, robustness,
stability with respect to the errors in the input data etc.
These methods validate the numerical results with respect to
the mathematical model and not with the physical processes;
this type of validation is absolutely necessary to ensure
the mathematical consistency of the method.

The external validation methods assume a comparison of the
numerical data with measured real data.  The main advantage
of these methods is that a good consistency of data
validates both the numerical data and the mathematical
model.  In the absence of measured data, one can do a
qualitative analysis: the evolution given by the numerical
model is similar to the observed one, without pretending
quantitative estimations.

\subsection{Internal validation}
We compare numerical results given by a 1-D version of our
model with the analytical solution for a Riemann Problem
\footnote{Ion S, Marinescu D, Cruceanu SG. 2015. Riemann
  Problem for Shallow Water Equations with Porosity. {\\ \tt
    http://www.ima.ro/PNII\_programme/ASPABIR/pub/slides-CaiusIacob2015.pdf}}.
Figure \ref{fig_1Dcomparison} shows a very good agreement
when the porosity is constant and a good one when the
porosity (cover plant density) varies.
\begin{figure}[htbp]
  \centering
  \includegraphics[width=0.6\textwidth]{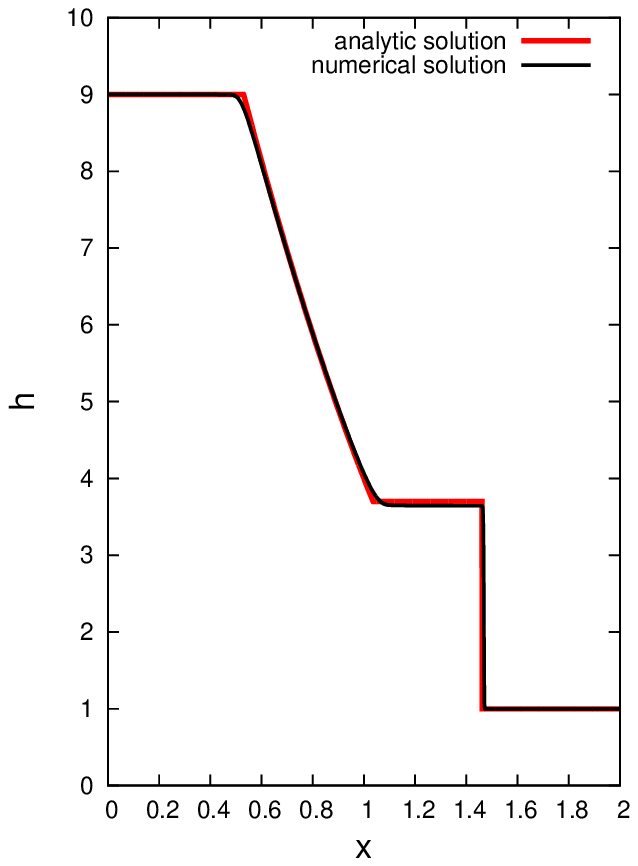}
  \hspace{-3cm}
  \includegraphics[width=0.6\textwidth]{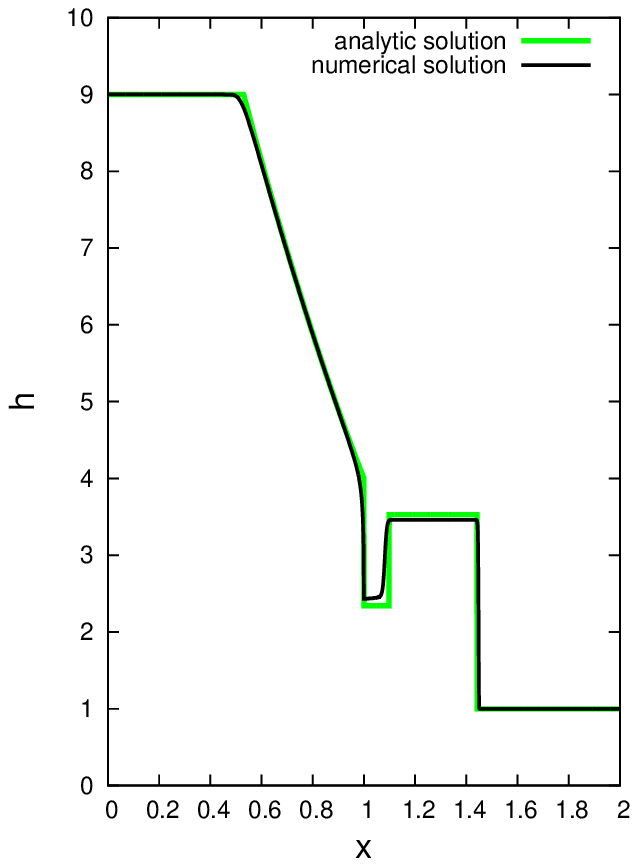}
  \caption{Comparison of the numerical and analytical
    solutions for the Riemann Problem.  The surface is
    described by $z=1$; at the initial moment we have the
    velocity field $\boldsymbol{v}=\boldsymbol{0}$ and a
    discontinuity in the water-depth $h$:
    $\left\{ h=9,\; {\rm for }\; x<1 \right\}$,
    $\left\{ h=1,\; {\rm for }\; x>1 \right\}$. Left picture
    - constant porosity: $\theta=1$. Right picture -
    variable porosity:
    $\left\{ \theta=0.8,\; {\rm for }\; x<1 \right\}$,
    $\left\{ \theta=1,\; {\rm for }\; x>1 \right\}$.}
  \label{fig_1Dcomparison}
\end{figure}
Also, in Figure \ref{fig_2Dcomparison} we analyze the
response of our model to the variation of the parameters.

\subsection{External validation}
Unfortunately, we do not have data for the water
distribution, plant cover density and measured velocity
field in a hydrographic basin to compare our numerical
results with.  However, to be closer to reality, we have
used GIS data for the soil surface of Paul's Valley and
accomplished a theoretical experiment: starting with a
uniform water depth on the entire basin and using different
cover plant densities, we run our model, ASTERIX based on a
hexagonal cellular automaton \cite{sds-ADataPortingTool}.
Figure \ref{fig_2Dcomparison} shows that the numerical
results are consistent with direct observations concerning
the water time residence in the hydrographic basin.
\begin{landscape}
  \begin{figure}[ht]
    \centering
    \includegraphics[width=0.6\textwidth]{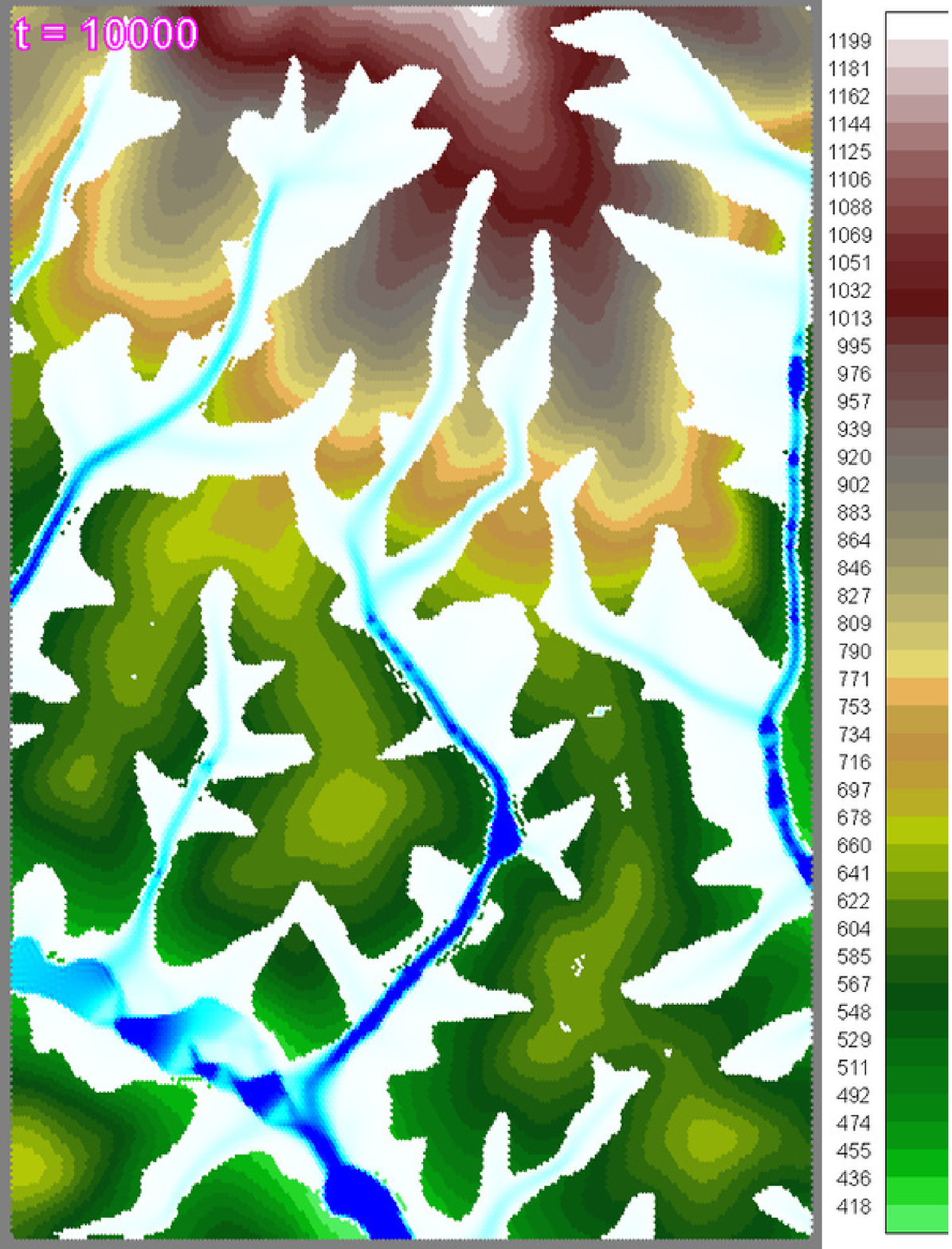}
    \hspace{1cm}
    \includegraphics[width=0.6\textwidth]{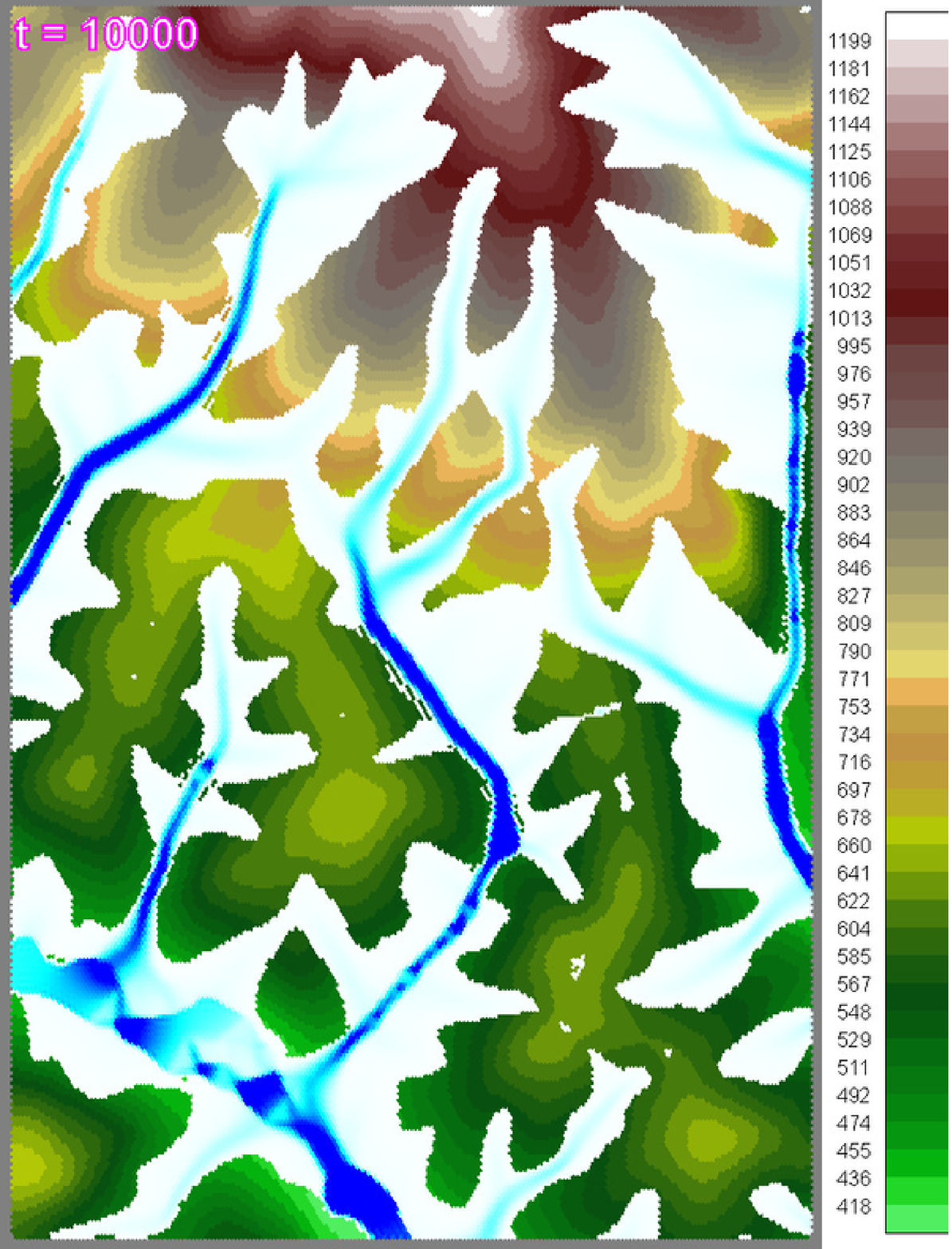}
    \caption{Snapshot of water distribution in Paul's Valley
      hydrographic basin.  Direct observations indicate that
      the water time residence depends on the density of the
      cover plant.  Our numerical data are consistent with
      terrain observations: the water drainage time is
      bigger for the case of higher cover plant density.
      $\theta=3\%$ and $\theta=35\%$ for the left and right
      picture, respectively.}
    \label{fig_2Dcomparison}
  \end{figure}
\end{landscape}

Figure \ref{fig_asterix_vs_caesar} shows the results for the
water content in Paul's Valley basin obtained with our
models ASTERIX and CAESAR-Lisflood-OSE
\cite{sds-ADataPortingTool, sds-ose}.  This variable $q$ is
in fact the relative amount of water in the basin at the
moment of time $t$:
\begin{equation*}
  q(t) = \displaystyle\frac{\displaystyle\int_{\Omega}h(t,x){\rm d}x}{\displaystyle\int_{\Omega}h(0,x){\rm d}x}.
\end{equation*}
\begin{figure}[htbp]
  \centering
  \includegraphics[width=0.49\textwidth]{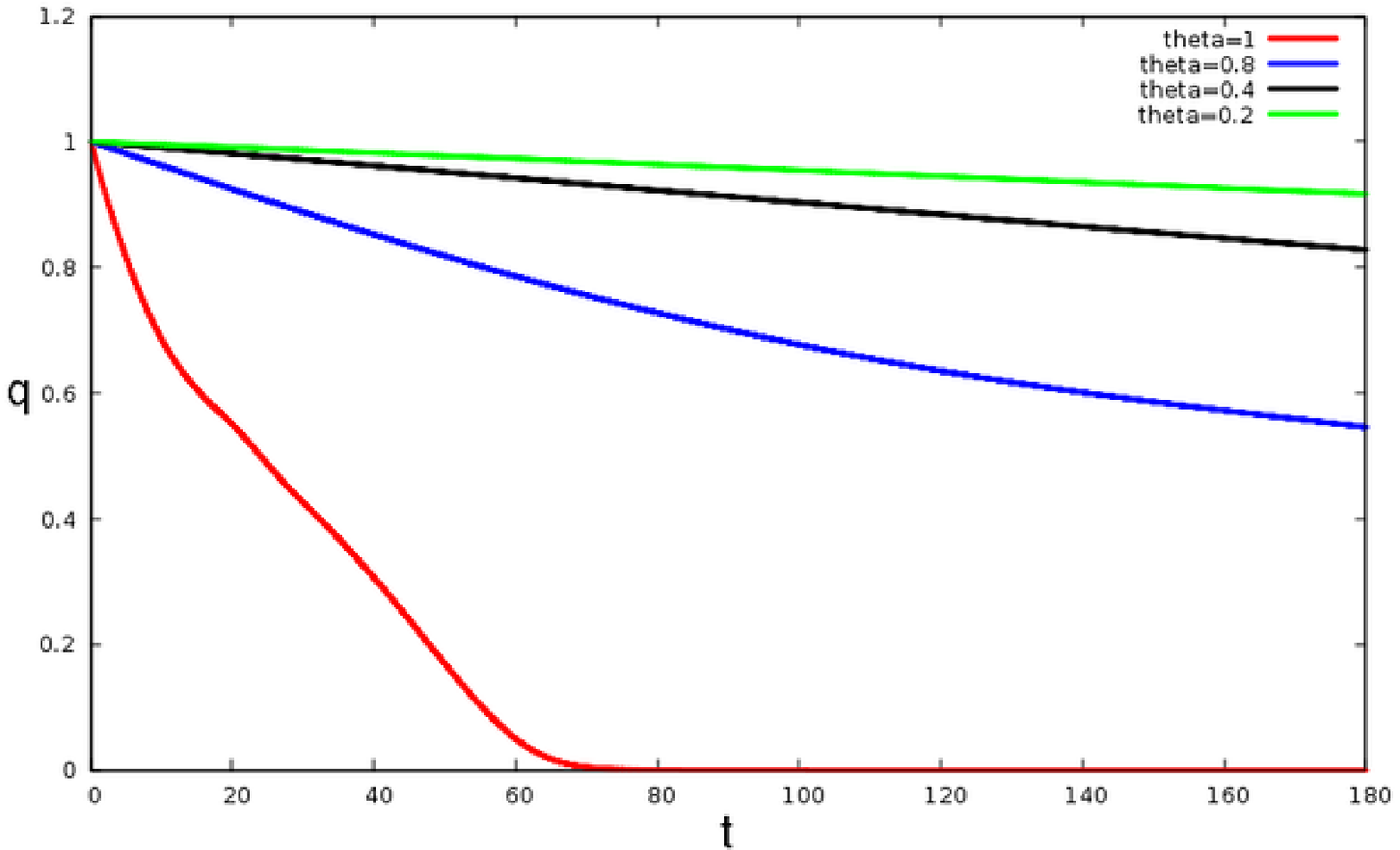}
  \includegraphics[width=0.49\textwidth]{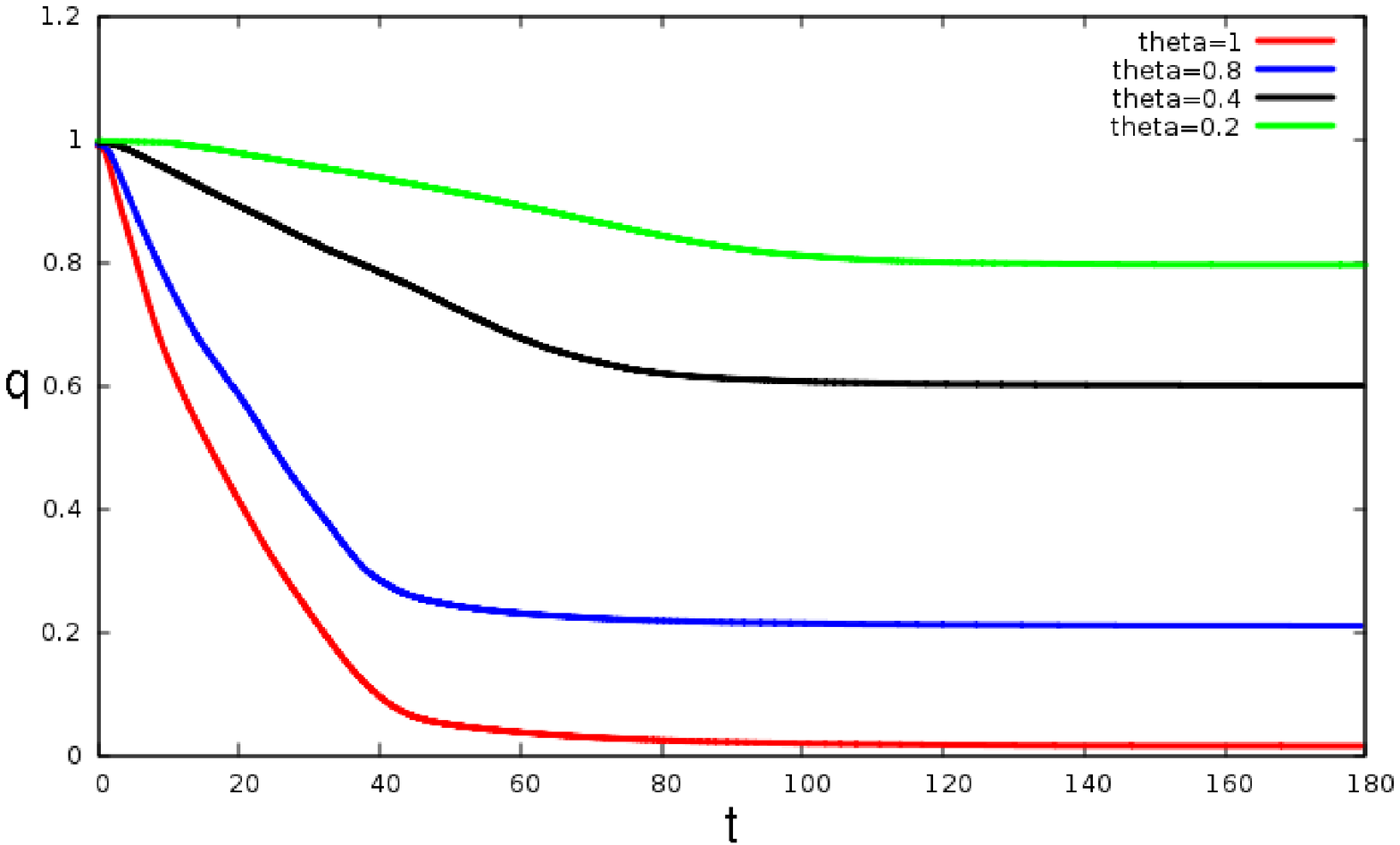}
  \caption{Time evolution of the water content in Paul's
    Valley hydrographic basin with ASTERIX (left picture)
    and CAESAR (right picture).}
  \label{fig_asterix_vs_caesar}
\end{figure}
This variable is also a measure of the amount of water
leaving the basin.  A general issue relates to whether
higher cover plant densities can prevent soil erosion and
flood or not.  Both pictures show that if the cover plant
density is increasing then the decreasing rate $\dot{q}$ of
$q$ is smaller.  We can think at a ``characteristic
velocity'' of the water movement in the basin and this
velocity is in a direct relation with $\dot{q}$.  We can now
speculate that smaller values of $\dot{q}$ imply softer
erosion processes.

This valley belongs to Ampoi's catchment basin.  Flood
generally appears when the discharge capacity of a river is
overdue by the water coming from the river catchment area.
Our pictures show that higher cover plant densities imply
smaller values of $\dot{q}$ which in turn give Ampoi River
the time to evacuate the water amount flowing from the
valley.

\section*{Acknowledgement}
Partially supported by the Grant 50/2012 ASPABIR funded by
Executive Agency for Higher Education, Research, Development
and Innovation Funding, Romania (UEFISCDI).

\bibliographystyle{plain}
\bibliography{biblio_arxiv_swmhpvh}

\end{document}